\title{\bf Injectives Hulls and Projective Covers in Categories of Generalized Uniform Hypergraphs}

\author{Martin Schmidt\\
	\small Department of Mathematics and Informatics\\[-0.8ex]
	\small Chiba University\\[-0.8ex] 
	\small Chiba, Japan
        }

\documentclass[12pt]{article}
\usepackage{amsmath,amssymb,mystyle,amsfonts,amsthm,soul,bbold, fancyhdr,quotes, graphicx, fix-cm}
\usepackage[top=2.5cm,bottom=2.5cm,left=3.5cm,right=3.5cm]{geometry}

\makeatother
\theoremstyle{plain}
\newtheorem{thm}{Theorem}[section]
\newtheorem{lemma}[thm]{Lemma}
\newtheorem{proposition}[thm]{Proposition}
\newtheorem{corollary}[thm]{Corollary}

\theoremstyle{definition}
\newtheorem{definition}[thm]{Definition}
\newtheorem{example}[thm]{Example}

\begin{document}
\newpage
\maketitle

\begin{abstract}
	We construct injective hulls and projective covers in categories of generalized uniform hypergraphs which generalizes the constructions in the category of quivers and the category of undirected graphs. While the constructions are not functorial, they are "sub-functorial", meaning they are subobjects of functorial injective and projective refinements. 
	
\end{abstract}

\section{Introduction}

We introduce categories of $(X,M)$-graphs where $M$ is a monoid and $X$ is a right $M$-set. An $(X,M)$-graph $G$ consists of a set of arcs $G(A)$ and a set of vertices $G(V)$ such that each arc is incident to $\#X$-vertices (multiplicities allowed) and $M$ informs the type of cohesivity between the vertices. The categories of $(X,M)$-graphs introduced in this paper are able to describe the types of incidence in the various definitions of graphs and hypergraphs by taking the monoid $M$ to be a submonoid of endomaps on a set $X$. Thus, when $X$ is a two-element set, the categories of $(X,M)$-graphs generalize the various categories of graphs and undirected graphs found in \cite{rB}. 

By separating syntax ($(X,M)$-graph theories) from semantics ($(X,M)$-graph categories), functorial constructions between $(X,M)$-graph categories are induced from the appropriate morphisms between theories. In particular, the constructions of injective hulls and projective covers (Section \ref{S:InjProj}) can be obtained by using the obvious morphisms of monoid actions as well as obvious interpretations of $(X,M)$-graph theories. 

\section{Categories of $(X,M)$-Graphs}\label{S:XMG}
We begin with a definition.

\begin{definition}\label{D:XMGraph} \mbox{}
	Let $M$ be a monoid and $X$ a right $M$-set. 
	The theory for $(X,M)$-graphs, $\DD G_{(X,M)}$, has two objects $V$ and $A$ with homsets given by
	\[
	\DD G_{(X,M)}(V,A) \defeq X,\ \ \
	\DD G_{(X,M)}(A,V) \defeq \empset,\ \ \
	\DD G_{(X,M)}(V,V)\defeq \{\Id_V\},\ \ \
	\DD G_{(X,M)}(A,A)\defeq M.
	\] 
	Composition is defined as  $m\circ x=x.m$ (the right-action via $M$), $m\circ m'=m'm$ (monoid operation of $M$).
	The category of $(X,M)$-graphs is defined to be the category of presheaves $\widehat{\DD G}_{(X,M)}\defeq [\DD G_{(X,M)}^{op},\Set]$. 
\end{definition}

We represent the theory for $(X,M)$-graphs and and reflexive $M$-graphs as follows.  
\[\xymatrix{ V  \ar[r]|-{X} & A \ar@(ur,dr)^M}
\] By definition, an $(X,M)$-graph $G\colon \DD G^{op}_{(X,M)}\to \Set$ has a set of vertices $G(V)$ and a set of arcs $G(A)$ along with right-actions for each morphism in $\DD G_{(X,M)}$. For example, $x\colon V\to A$ in $\DD G_{(X,M)}$ yields a set map $G(x)\colon G(A)\to G(V)$ which takes an arc $\alpha\in G(A)$ to $\alpha.x\defeq G(x)(\alpha)$ which we think of as its $x$-incidence.\footnote{Note that we use the categorical notation of evaluation of a presheaf as a functor for the set of vertices $G(V)$ and set of arcs $G(A)$ rather than the conventional graph theoretic $V(G)$ and $E(G)$ for the vertex set and edge set.} For an element $m$ in the monoid $M$, the corresponding morphism $m\colon A\to A$ in $\DD G_{(X,M)}$ yields a right-action $\alpha.m\defeq G(m)(\alpha)$ which we think of as the $m$-associated partner of $\alpha$. 

Each $(X,M)$-graph $G$ induces a set map $\partial_G\colon G(A)\to G(V)^X$ such that $\partial_G(\alpha)\colon X\to G(V)$ is the parametrized incidence of $\alpha$, i.e., $\alpha.x=\partial_G(\alpha)(x)$. The $x$-incidence can be recovered from a parametrized incidence by precomposition of the map $\named x\colon 1\to X$ which names the element $x$ in $X$. Observe that the $m$-associated partner of an arc $\alpha$ in $G$ has the parametrized incidence such that the following commutes
\[
\xymatrix@C=3em{ X \ar@/_1em/[rrr]_{\partial_{G}(\alpha.m)} \ar[r]^-{\langle \Id_X, \named m \rangle} & X\x M \ar[r]^-{\text{\tiny action}} & X \ar[r]^-{\partial_{G}(\alpha)} & G(V).}
\]

\begin{example}
	Let $X$ be a set and $M$ a submonoid of endomaps $\End(X)$. The right-action of $M\subseteq \End(X)$ on $X$ is given by evaluation, e.g. $x.f\defeq f(x)$.  When $X$ is a two-element set the categories $\widehat{\DD G}_{(\{\Id_X\},X)}$ and $\widehat{\DD G}_{(\Aut(X),X)}$ are the categories of quivers \cite{wG} and symmetric directed graphs \cite{rB}.
\end{example}

\subsection{Nerve-Realization Adjunctions}
The symbols and notation in this section follow from \cite{eR}. 

Let $I\colon \DD T \to \C M$ be functor from a small category $\DD T$ to a cocomplete category $\C M$. Since the Yoneda embedding $y\colon \C T\to \widehat{\DD T}$ is the free cocompletion of a small category there is a essentially unique adjunction $R\dashv N\colon \C M\to \widehat{\DD T}$, called the nerve-realization adjunction, such that $Ry\iso I$. 
\[
\xymatrix{ \DD T \ar[dr]_I \ar[r]^{y} & \widehat{\DD T} \ar@<-.4em>[d]_{R}^{\dashv} \ar@{<-}@<+.5em>[d]^{N}\\ & \C M }
\]
The nerve and realization functors are given on objects by $N(m)=\C M(I(-), m)$,  $R(X)=\colim_{(c,\varphi)\in\int F} I(c)$ respectively,
where $\int F$ is the category of elements of $X$ (\cite{hA}, Section 2, pp 124-126).\footnote{In \cite{hA}, the nerve functor is called the singular functor.} 

We call a functor $I\colon \DD T \to \C M$ from a small category to a cocomplete category an interpretation functor. The category $\DD T$ is called the theory for $I$ and $\C M$ the modeling category for $I$. An interpretation $I\colon \DD T\to \C M$ is dense, i.e., for each $\C M$-object $m$ is isomorphic to the colimit of the diagram
$I\ls m \to \C M,\ (c,\varphi)\mapsto I(c),$
if and only if the nerve $N\colon \C M\to \widehat{\DD T}$ is full and faithful (\cite{sM}, Section X.6, p 245). When the right adjoint (resp. left adjoint) is full and faithful we call the adjunction reflective (resp. coreflective).\footnote{since it implies $\C M$ is equivalent to a reflective (resp. coreflective) subcategory of $\widehat{\DD T}$}

A functor $F\colon \C C\to \C C'$ between small categories $\C C$ and $\C C'$ induces an essential geometric morphism $F_!\dashv F^*\dashv F_*\colon \widehat{\C C}\to \widehat{\C C'}$ where $F_!$ is the realization of $y_{\C C'}\circ F$, $F^*$ is the nerve of $y_{\C C'}\circ F$ and $F_*$ is the nerve of $F^*\circ y_{\C C'}$ where $y_{\C C}\colon \C C\to \widehat{\C C}$ and $y_{\C C'}\colon \C C'\to \widehat{\C C'}$ are the Yoneda embeddings (see \cite{mR} pp 194-198, Section \ref{S:Category}). On objects $W$ in $\C C$ and $Z$ in $\C C'$, the functors are given by
\begin{align*}
&  \textstyle F_!(W)\defeq \colim_{(C,c)\in\int W} y_{\C C'}F(C),\\
&  \textstyle F^*(Z)\defeq \widehat{\C C'}(y_{\C C'}F(-),Z),\\ &  \textstyle F_*(W)\defeq \widehat{\C C}(F^*y_{\C C'}(-),W)
\end{align*}
where $\int A$ is the category of elements for $A$ (\cite{eR}, Section 2.4). In the subsequent, we denote the representables for the vertex and arc object by $\underline V$ and $\underline A$ respectively. 

Consider the $(X,M)$-graph theory $\DD G_{(\empset, 1)}$, i.e., the discrete category with two objects $V$ and $A$. Then for an $(X,M)$-graph theory $\DD G_{(X,M)}$, there is the inclusion functor $\iota\colon \DD G_{(\empset,1)}\to \DD G_{(X,M)}$. Thus there is an essential geometric morphism $\iota_!\dashv \iota^*\dashv \iota_*\colon \Set^2\to \widehat{\DD G}_{(X,M)}$. The $\iota$-extension $\iota_!$ takes the pair of sets $(S(V),S(A))$ to the coproduct $\bigsqcup_{S(V)}\underline V\sqcup \bigsqcup_{S(A)}\underline A$ since the category of elements for $(S(V),S(A))$ lacks internal cohesion. The $\iota$-restriction $\iota^*$ takes an $(X,M)$-graph $G$ to the pair of sets $(G(V),G(A))$. By \cite{eR} (Proposition 3.3.9), it creates all limits and colimits.  The $\iota$-coextension $\iota_*$ sends $(S(V),S(A))$ to the $(X,M)$-graph with vertex set $S(V)$ and arc set $\Set^2((X,|M|),(S(V),S(A)))=S(V)^X\x S(A)^{|M|}$ where $|M|$ is the underlying set of $M$. The right-actions are given by $(f,s).x=f(x)$, $(f,s).m=(f\circ m,s.m)$ where $m\colon X\to X$ is the right-action map by $m\in M$ and $s.m\colon |M|\to S(A)$ is defined $s.m(m')\defeq s(mm')$. 

The counit $\varepsilon\colon \iota_!\iota^*\Rightarrow \Id$ of the adjunction $\iota_!\dashv \iota^*$ on a component $\varepsilon_G\colon \bigsqcup_{G(V)}\underline V\sqcup \bigsqcup_{G(A)}\underline A \to G$ is the epimorphism induced by the classification maps $v\colon \underline V\to G$ and $\alpha\colon \underline A\to G$ for vertices $v\in G(V)$ and arcs $\alpha\in G(A)$. Therefore, the $\iota$-restriction functor is faithful and thus $\iota^*\colon \widehat{\DD G}_{(X,M)}\to \Set^2$ is monadic (\cite{mP}, p. 227)

The unit $\eta\colon \Id\Rightarrow \iota_*\iota^*$ of the adjunction $\iota^*\dashv \iota_*$ on a component $\eta_G\colon G\to \iota_*\iota^*(G)$ is the identity on vertices and sends arc $\alpha\in G(A)$ to $(\partial_G(\alpha),\overline \alpha)$ where $\overline \alpha\colon |M|\to G(A)$ is the constant map. Thus for each $(X,M)$-graph $G$ the component $\eta_G$ is a monomorphism.

\section{Injective and Projective $(X,M)$-Graphs}\label{S:InjProj}

We set $\Proj\defeq \iota_!\iota^*$ and $\Inj\defeq \iota_*\iota^*$ where $\iota\colon \DD G_{(\empset,X)}\to \DD G_{(X,M)}$ is the functor given above. Then since adjunctions are closed under composition, we have
\[
\Proj \dashv \Inj\colon \widehat{\DD G}_{(X,M)}\to \widehat{\DD G}_{(X,M)}.
\]
We will show in this section that the natural transformations $\varepsilon\colon \Proj \Rightarrow \Id$ and $\eta\colon \Id \Rightarrow \Inj$ can be thought of as the functorial projective and injective refinements for non-initial  $(X,M)$-graphs.

We first characterize the class of injective and projective objects. Recall that an object $Q$ in a category $\C C$ is injective provided for each monomorphism $m\colon A\to B$ and morphism $f\colon A\to Q$ there exists a morphism (not necessarily unique) $k\colon B\to Q$ such that $f=km$. Dually, an object $P$ in $\C C$ is (regular) projective\footnote{Note that since regular epimorphisms are equivalent to epimorphisms in categories of presheaves, a regular projective object is equivalent to a projective object.} provided for each (regular) epimorphism $e\colon B\to A$ and each morphism $f\colon P\to A$ there is a morphism $k\colon P\to B$ such that $f=ek$.\footnote{The results in this section generalize the results of \cite{wG} and \cite{kW}.}
\begin{proposition}\label{P:Injective}
	A  $(X,M)$-graph $Q$ is injective if and only if $Q$ is non-initial and for each set map $f\colon X\to Q(V)$, there is an arc $\alpha\in Q(A)$ such that the incidence map $\partial_Q(\alpha)$ is equal to $f$. 
\end{proposition}
\begin{proof}
	Suppose $Q$ is injective and consider the set map $f\colon X\to Q(V)$. This is equivalent to giving an $(X,M)$-graph morphism $\overline f\colon \bigsqcup_{x\in X}\underline V\to Q$. Consider the inclusion $m\colon \bigsqcup_{x\in X}\underline V\to \underline A$ induced by the morphisms $\underline x\colon \underline V\to \underline A$. Since $Q$ is injective, there is a morphism $\alpha\colon \underline A\to I$ such that  $\alpha m=\overline f$. By Yoneda, this is equivalent to an arc $\alpha\in I(A)$ with incidence map $\partial_I(\alpha)=f$.
	
	Conversely, let $f\colon G\hookrightarrow H$ be a monomorphism and $g\colon G\to Q$ a morphism of  $(X,M)$-graphs. Since $Q$ is non-initial, there is a vertex $v\in Q(V)$. Each arc $\alpha$ in $H$ has incidence $\partial_H(\alpha)\colon X\to H(V)\iso f_V(G(V))\sqcup H(V)\backslash f_V(G(V)$ where $f_V(G(V))$ is the image of the vertices in $G$ under $f$. For each arc $\alpha$ in $H$ not in the image of $f_A$, let $j_\alpha\colon X\to Q(V)$ be the set map $[g_V,!]\circ \partial_G(\alpha)$ given by universal property of the disjoint union
	\[
	\xymatrix{ && f_V(G(V))\iso \inv f_V(G(V)) \ar@{>->}[d] \ar[rr]^-{g_V} && g_V(G(V)) \ar@{>->}[d] \\ X \ar[rr]^-{\partial_G(\alpha)} && f_V(G(V))\sqcup H(V)\backslash f_V(G(V)) \ar[rr]^-{[g_V,!]} && Q(V) \\ 
		&& H(V)\backslash f_V(G(V)) \ar@{>->}[u] \ar[rr]^-{!} && \{v\} \ar@{>->}[u]}
	\] 
	Thus by assumption, we may choose an arc $[\alpha]\in Q(A)$ with incidence equal to $j_\alpha$.  We define the following maps $h_V\colon H(V)\to Q(V)$ and $h_A\colon H(A)\to Q(A)$
	\[
	h_V(w)\defeq \begin{cases} g_V(u) & \text{if }\exists u\in G(V), f_V(u)=w \\ v & \text{if }\forall u\in G(V), f_V(u)\neq w\end{cases}\]
	\[	h_A(\alpha)\defeq \begin{cases} g_A(\beta) & \text{if }\exists \beta\in G(A), f_A(\beta)=\alpha \\ [\alpha] & \text{if }\forall \beta\in G(A), f_A(\beta)\neq \alpha \end{cases}
	\]
	By construction this defines a morphism $h\colon H\to Q$ such that $h\circ f=g$. Therefore $Q$ is injective.
	
\end{proof}

\begin{corollary}
	The class of injective objects in $\widehat{\DD G}_{(X,M)}$ is precisely the class of non-initial split subobjects\footnote{An object $H$ is a split subobject of $G$ provided it admits a split monomorphism $s\colon H\to G$.} of objects in the essential image of the functor $\Inj$.
\end{corollary}
\begin{proof}
	Let $Q$ be an injective object in $\widehat{\DD G}_{(X,M)}$.  Hence $Q$ is non-initial and thus by the previous lemma, we have $\Inj(Q)$ is an injective object. Then since $\eta_Q\colon Q \to \Inj(Q)$ is a monomorphism, there must be a split epimorphism $r\colon \Inj(Q)\to Q$ such that $r\eta_Q=\Id$ by the property of $Q$ being injective.  
	
\end{proof}

We also have the dual argument that the class of projective objects in the category of  $(X,M)$-graphs is precisely the split quotients of objects in the essential image of $\Proj$. 

\begin{proposition}
	A  $(X,M)$-graph $P$ is projective if and only if it is a coproduct of representables. 
\end{proposition}
\begin{proof}
	Suppose $P$ is projective. Since $\varepsilon_P\colon \Proj(P)\to P$ is an epimorphism, there must exist a section $s\colon P\to \Proj(P)$ such that $\varepsilon_Ps=\Id$ by the property that $P$ is projective. Since $\Proj(P)$ is a coproduct of representables and $s$ is a split monomorphism, $P$ must also be a coproduct of representables. Conversely, representables $\underline V$ and $\underline A$ in a category of presheaves are always projective. Since projective objects are closed under coproducts, the reverse condition is also true.
	
\end{proof}
\begin{corollary}
	The class of projective objects in $\widehat{\DD G}_{(X,M)}$ is the object class of the essential image of the functor $\Proj$. 
\end{corollary}
\begin{proof}
	Given a coproduct of representables $\bigsqcup_S \underline V\sqcup \bigsqcup_T \underline A$, let $H$ be the $(X,M)$-graph with vertex set $H(V)\defeq S$ and arc set $H(A)=T$. Take some $s\in S$ and define right-actions $t.x\defeq s$ and $t.m=t$ for each $t\in T$, $x\in X$ and $m\in M$. Then $\Proj(H)\iso \bigsqcup_S \underline V\sqcup \bigsqcup_{T}\underline A$. 
\end{proof}

Next, we construct injective hulls and projective covers for  $(X,M)$-graphs. Recall that a monomorphism $i\colon G\to \widetilde G$ is essential provided for each morphism $h\colon G\to H$ such that $hi$ is a monomorphism implies $h$ is a monomorphism. An injective hull of an object $G$ is an essential monomorphism $i\colon G\to \widetilde G$ where $\widetilde G$ is injective. Dually, an epimorphism $e\colon \overline G\to G$ is essential provided for each morphism $h\colon H\to G$ such that $eh$ is an epimorphism implies $h$ is an epimorphism. A projective cover of an object $G$ is an essential epimorphism $e\colon \overline G\to G$ where $\overline G$ is projective.

In the case of the initial  $(X,M)$-graph $0$, it is straightforward to verify the terminal morphism $0\to 1$ is the injective hull. For a non-initial  $(X,M)$-graph $G$ we define $\widetilde G$ to be the $(X,M)$-graph with vertex set $\widetilde G(V)\defeq G(V)$ and arcs set $\widetilde G(A)\defeq G(A)\sqcup \setm{f\colon X\to G(V)}{\forall \alpha\in G(A),\ \partial_G(\alpha)\neq f}$ with the obvious right-action. Then $\widetilde G$ is an injective object and there is an obvious inclusion $i\colon G\to \widetilde G$. To show that it is essential, let $h\colon \widetilde G\to H$ be a morphism such that $hi$ is a monomorphism. Then since $i$ is bijective on vertices, $h_V$ must be injective. On arcs, it is enough to show that $h_A$ is injective on $\widetilde G(A)\backslash G(A)$. However, this is trivial since there is only one arc $f\in \widetilde G(A)\backslash G(A)$ with incidence $f\colon X\to G(V)$. Hence $h$ is a monomorphism.

For the projective cover, we define $\overline G\defeq \bigsqcup_S \underline V \sqcup \bigsqcup_T \underline A$ where \\$S\defeq \setm{v\in G(V)}{\forall \alpha\in G(A),\forall x\in X,\ \alpha.x\neq v}$, i.e., $S$ is the set of isolated vertices in $G$, and $T$ is a generating subset of $G(A)$ for the right $M$-action $G(A)\x M\to G(A)$ of minimal cardinality, i.e., for each $\alpha\in G(A)$ there exists a $\beta\in T$ and an element $m\in M$ such that $\beta.m=\alpha$.  Since $T$ generates $G(A)$ under the right-action of $M$ and $S$ is the set of vertices of $G$ which are not incident to an arc, the restriction of $\varepsilon_{G|\overline G}\colon \overline G\to G$ is an epimorphism. It is clear that if $h\colon H\to \overline G$ is a morphism such that $he$ is an epimorphism, then $h$ must be an epimorphism since $\overline G$ is a coproduct of representables of minimal size. 
\[
\xymatrix@!=.3em{ & \widetilde G \ar@{>->}[dr] & \\ G \ar@{>->}[ur] \ar@{>->}[rr]^-{\eta_G} && \Inj(G)} \qquad \qquad
\xymatrix@!=.3em{ & \overline G \ar@{>->}[dl] \ar@{->>}[dr] & \\ \Proj(G) \ar@{->>}[rr]^-{\varepsilon_G} && G }  
\]
Note that these assignments $\widetilde G$ and $\overline G$ do not extend to functors since there is choice involved. However, by construction we see that both the injective hull $\widetilde G$ and projective cover $\overline G$ embed into $\Inj(G)$ and $\Proj(G)$ which are functorial constructions.

\section{Applications}\label{S:Category}
\label{S:Interp}

We follow the definition given in \cite{cJ}. 

\begin{definition}
	Let $F\colon \Set\to \Set$ be an endofunctor. The category of $F$-graphs $\C G_F$ is defined to be the comma category $\C G_F\defeq \Set\ls F$.
\end{definition}
In other words, an $F$-graph $G=(G(E),G(V),\partial_G)$ consists of a set of edges $G(E)$, a set of vertices $G(V)$ and an incidence map $\partial_G\colon G(E)\to F(G(V))$. A morphism $(f_E,f_V)\colon (G(E),G(V),\partial_G)\to (H(E),H(V),\partial_H)$ is a pair of set maps $f_E\colon G(E) \to H(E)$ and $f_V\colon G(V)\to H(V)$ such that the following square commutes
\[
\xymatrix{ G(E) \ar[d]_{\partial_G} \ar[r]^{f_E} & H(E) \ar[d]^{\partial_H} \\ F(G(V)) \ar[r]^{F(f_V)} & F(H(V)). }
\]
It is well-known that the category of $F$-graphs is cocomplete with the forgetful functor $U\colon \C G_F\to \Set\x \Set$ creating colimits \cite{cJ}.

Let $\DD G_{(X,M)}$ be a theory for $(X,M)$-graphs and $q$ an element in $F(X)$ such that $F(m)(q)=q$ for each $m\in M$ where $m\colon X\to X$ is the right-action map. We define $I(V)\defeq (\empset,\ 1,\ !_1),$ and $I(A)\defeq (1,\ X,\ \named q)$ where $!_1\colon \empset \to 1$ is the initial map and $\named x\colon 1\to X$ the set map with evaluation at $x\in X$. On morphisms, we set
\begin{align*} 
(x\colon V\to A) \quad &\mapsto \quad  I(x)\, \defeq \ (!_1,\named x)\colon (\empset,\ 1,\ !_1)\to (1,\ X,\ \named q),\\
(m\colon A\to A) \quad &\mapsto \quad I(m)\defeq (\Id_1, F(m))\colon (1,X,\named q)\to (1,X,\named q).
\end{align*}
Verification that $I\colon \DD G_{(X,M)}\to \C G_F$ is a well-defined interpretation functor is straightforward.

\subsection{The Category of Hypergraphs}\label{S:PGraphs}

We recall that a hypergraph $H=(H(V),H(E),\varphi)$ consists of a set of vertices $H(V)$, a set of edges $H(E)$ and an incidence map $\varphi\colon H(E)\to \C P(H(V))$ where $\C P\colon \Set\to \Set$ is the covariant power-set functor. In other words, we allow infinite vertex and edge sets, multiple edges, loops, empty edges and empty vertices.\footnote{An empty vertex is a vertex  not incident to any edge in $H(E)$. An empty edge is an edge $e$ such that $\varphi(e)=\empset$.} In other words the category of hypergraphs $\C H$ is the category of $\C P$-graphs.

Let $X$ be a set and apply the definition for the interpretation given above in \ref{S:Interp} for $\DD G_{(X,\Aut(X))}$ with $q\defeq X$ in $\C P(X)$. Note that for each automap $\sigma\colon X\to X$, $\C P(\sigma)$ is the identity map. Thus the interpretation $I\colon \DD G_{(X,\Aut(X))}\to \C H$ defined in Section \ref{S:Interp} is a well-defined functor. 

The nerve $N\colon \C H\to \widehat{\DD G}_{(X,\Aut(X))}$ induced by $I$ takes a hypergraph \\$H=(H(E),H(V),\varphi)$ to the $(X,\Aut(X))$-graph $N(H)$ with vertex and arc set given by
\begin{align*}
N(H)(V)&=\C H(I(V),H)=H(V),\\
N(H)(A)&=\C H(I(A),H)=\setm{(\beta,f)\in H(E)\x H(V)^X}{\C P(f)=\varphi(\beta)}
\end{align*} Notice that in the case a hyperedge $e$ has less than $\#X$ incidence vertices the nerve creates multiple edges and if a hyperedge has more than $\#X$ incidence vertices there is no arc in the correponding $(X,\Aut(X))$-graph given by the nerve (see Example \ref{E:HyperNR} below).

The realization $R\colon \widehat{\DD G}_{(X,\Aut(X))}\to \C H$ sends a $(X,\Aut(X))$-graph $G$ to the hypergraph $R(G)=(R(G)(E),R(G)(V),\psi)$ with vertex, edge sets and incidence map given by
\begin{align*}
& R(G)(V)=G(V),\\  & R(G)(E)=G(A)/\sim, \quad \text{($\sim$ induced by $\Aut(X)$)},\\
&\psi\colon R(G)(E)\to \C P(R(G)(V)), \quad [\gamma]\mapsto \setm{v\in G(V)}{\exists x\in X, \gamma.x=v}
\end{align*} 
For a $(X,\Aut(X))$-graph morphism $f\colon G\to G'$, the hypergraph morphism \\$R(f)\colon R(G)\to R(G')$ has $R(f)_V\defeq f_V$ and $R(f)_E\defeq [f_A]$ where $[f_A]\colon \frac{G(A)}{\sim} \to \frac{G'(A)}{\sim}$ is induced by the quotient.

\begin{example}\label{E:HyperNR} \mbox{}
	\begin{enumerate}
		\item \label{E:Creates} Let $X=\{a,b,c\}$ and consider the hypergraph $H$ with two vertices $0$ and $1$ and one hyperedge $\alpha$ between them. Then the nerve $N(H)$ has two vertices $0$ and $1$ and arc set $N(H)(A)=\{001, 010, 100, 011, 101, 110 \}$ with a pair of $\Aut(X)$-partners $001\sim 010\sim 100$ and $011\sim 101\sim 110$. The realization identifies the $\Aut(X)$-partners and thus $RN(H)$ has two vertices $0$ and $1$ and two edges $[001]$ and $[011]$ between them. The counit $\varepsilon_H\colon RN(H)\to H$ is bijective on vertex set and sends $[001]$ and $[011]$ to $\alpha$.
		\item Let $X=\{a,b,c\}$ and consider the hypergraph $H$ with four vertices and one hyperedge $\alpha$ connecting them. The nerve $N(H)$ has vertex set equal to $H(V)$ but empty arc set $N(H)(A)=\empset$.  The counit $\varepsilon_H\colon RN(H)\to H$ is the inclusion of vertices.
	\end{enumerate}
\end{example}

We are able to use the adjunction above to classify the projective objects in the category of hypergraphs. Recall that a right adjoint functor is faithful if and only if the counit is an epimorphism.

\begin{lemma}
	Let $X$ be a set with cardinality $\kappa$ greater than $1$. Then the nerve $N\colon \C H\to \widehat{\DD G}_{(X,\Aut(X))}$ of the interpretation $I\colon \DD G_{(X,\Aut(X))}\to \C H$ is faithful on the full subcategory $\C H_\kappa$ consisting of hypergraphs $H$ with $\max_{e\in H(E)}\varphi(e)$ at most of cardinality $\kappa$.
\end{lemma}
\begin{proof}
	Given hypergraph $H$ in the subcategory $\C H_{\kappa}$ the counit $\varepsilon_H\colon RN(H)\to H$ of the adjunction $R\dashv N\colon \C H\to \widehat{\DD G}_{(X,\Aut(X))}$ at component $H$ is an epimorphism since for each hyperedge $e\in H(E)$, there is a $(X,\Aut(X))$-graph morphism $f\colon I(A)\to H$ such that $f_E$ takes the lone hyperedge in $I(A)$ to $e$. 
\end{proof}

Let $I(V)$ be the hypergraph with one vertex and no hyperedges. For each cardinal number $k$, let $E_k\defeq I(A)$ be the hypergraph where $I\colon \sG_k\to \C H$ is the interpretation functor above. 

\begin{lemma}
	The proper class of objects consisting of the vertex object $I(V)$ and hyperedge objects $(E_k)_{k\in \Set}$ is a family of separators for the category of hypergraphs.  
\end{lemma}
\begin{proof}
	Let $f, g\colon H\to H'$ be distinct hypergraph morphisms. Let $\kappa$ be the maximum of the cardinalities such that $H$ and $H'$ are in the subcategory $\C H_{\kappa}$. By the lemma above, the nerve $N\colon \C H\to \widehat{\DD G}_{(X,\Aut(X))}$ is faithful on $\C H_\kappa$. Thus we have $N(f)\neq N(g)$. Therefore, either $I(V)$ separates $f$ and $g$ or $I(\underline A)=E_k$ separates $f$ and $g$ by definition of the nerve functor. 
\end{proof}

\begin{proposition}
	A hypergraph $H$ is (regular) projective if and only if it has no hyperedges.  
\end{proposition}
\begin{proof}
	If a hypergraph has no edges it is projective since it is a coproduct of the vertex object $I(V)$ which is clearly projective. Conversely, by the lemma above it is enough to show that each hyperedge object $E_k$ is not projective. Let $1$ be the terminal hypergraph with one vertex and one hyperedge. Then every morphism $E_k\to 1$ is a (regular) epimorphism. Let $E_r$ be the hyperedge object with $r$ vertices where $r$ is of cardinality strictly greater than $k$. Since there is no morphism from $E_k$ to $E_r$ there is no factorization of $E_k\to 1$ through $E_r$ showing $E_k$ is not (regular) projective.
\end{proof}

For each set $X$, the interpretation functor $I\colon \DD G_{(X,\Aut(X))}\to \C H$ factors through the full subcategory $\C H_{k}$ of hypergraphs consisting of hypergraphs $H$ such that the incidence of each edge is of cardinality less than or equal to the cardinality $k$ of $X$. In other words, $\C H_k$ is the slice category $\Set\ls \C P_k$ where $\C P_k$ is the covariant $k$-power set functor which takes a set to the set of all subsets with cardinality less than or equal to $k$. The inclusion functor $i\colon \C H_k\hookrightarrow \C H$ admits a coreflector $r\colon \C H\to \C H_k$ which takes a hypergraph $H$ to the hypergraph $r(H)$ with vertex set $r(H)(V)=H(V)$ and edge set $r(H)(E)\defeq \setm{\alpha\in H(E)}{\#\varphi(\alpha)\leq k}$. Therefore the nerve realization for the interpretation $I_k\defeq rRy\colon \DD G_{(X,\Aut(X))}\to \C H_k$ is $R_k\dashv N_k\colon \C H_k\to \widehat{\DD G}_{(X,\Aut(X))}$ where $R_k=rR$ and $N_k=Ni$. Moreover, by restricting to the subcategory $\C H_k$ the counit $\varepsilon_H\colon rRNi(H)\Rightarrow H$ is now an epimorphism.

\begin{proposition}
	For a cardinal number $k$, the class of projective objects in $\C H_k$ are precisely the coproducts of $I_k(V)$ and $I_k(A)$ where $I_k\colon \DD G_{(X,\Aut(X))}\to \C H_k$ is the interpretation functor described above. 
\end{proposition}
\begin{proof}
	It is clear that $I_k(V)$ is projective. Since a hypergraph $H$ in $\C H_k$ has an edge if and only if it admits a morphism from $I_k(A)$ to it and since epimorphisms in $\C H_k$ are those morphisms surjective on vertex and edge sets, it is clear $I_k(A)$ is projective. Therefore the coproducts of $I_k(V)$ and $I_k(A)$ are projective. Conversely, consider the following composition 
	\[
	\xymatrix{ R_k(\Proj(N(H))) \ar@{->>}[r] & R_kN_k(H) \ar@{->>}[r]^-{\varepsilon_H} & H}
	\]
	where the morphism $\xymatrix{R_k(\Proj(N(H))) \ar@{->>}[r] & R_kN_k(H)}$ is the application of $R_k$ on the epimorphism $\Proj(N(H))\to N(H)$ described in Section \ref{S:InjProj}. Note that since $R_k$ is a left adjoint, it preserves epimorphisms. Moreover, $R_k$ preserves colimits, therefore $R_k(\Proj(N(H)))=\bigsqcup_{N(H)(V)}R_k\underline V\sqcup \bigsqcup_{N(H)(A)}R_k\underline A$. Then since $R_ky=I_k$ where $y$ is the Yoneda embedding, we have $R_k(\Proj(N(H)))=\bigsqcup_{N(H)(V)}R_k\underline I_k(V)\sqcup \bigsqcup_{N(H)(A)}I_k(A)$. Therefore, every object in $\C H_k$ admits an epimorphism from a projective object (i.e., $\C H_k$ has enough projectives). Thus every projective in $\C H_k$ is a split subobject of the essential image of the functor $R_k\Proj\colon \widehat{\DD G}_{(X,\Aut(X))}\to \C H_k$. However it is clear that the only split subobjects are coproducts of $I_k(V)$ and $I_k(A)$. 
	
\end{proof}

\begin{proposition}
	A hypergraph $Q$ is injective if and only if $Q$ is non-initial and for each subset of $S\subseteq Q(V)$, there is an edge $\alpha\in Q(E)$ with incidence equal to $S$.
\end{proposition}
\begin{proof}
	Suppose $Q$ is injective. For each subset $S\subseteq Q(V)$, let $E_S$ be the hypergraph with one edge $e$ with incident equal to $\varphi(e)=S$. Let $f\colon \bigsqcup_S I(V)\hookrightarrow E_S$ and $g\colon \bigsqcup_S I(V)\to Q$ be the inclusions of vertices. Since $Q$ is injective there is a morphism $h\colon E_S\to Q$ which necessarily is a monomorphism. Hence $Q$ must have an edge $q$ with incidence $\varphi(q)=S$. 
	
	Conversely, let $f\colon G\to H$ be a monomorphism and $g\colon G\to Q$ be a morphism in the category of hypergraphs. Since $Q$ is non-initial, there is a vertex $v\in Q(V)$. We define the morphism $h\colon H\to Q$ on vertices
	\[
	h_V(w)\defeq \begin{cases} g_V(u) & \text{if }\exists u\in G(V), g_V(u)=w \\ v & \text{if } \forall u\in G(V), g_V(u)\neq w. \end{cases}
	\]
	Each edge $e$ in $H$ not in the image of $f_A$ has incidence subset $S\subseteq H(V)$ which can be decomposed $S\iso S_0\sqcup S_1$ such that $S_0$ is in the image of $f_V$ and $S_1$ is disjoint to the image of $f_V$. Then for such an edge $e$ in $H(E)$, choose an edge $[e]\in Q(E)$ with incidence $g(f_V^{-1}(S_0))\cup \{v\}$ where $g_V(f_V^{-1}(S_0))$ is the image of $f_V^{-1}(S_0)$ under $g_V$. Then we define 
	\[
	h_E(e)\defeq \begin{cases} g_E(b) & \text{if }\exists b\in G(E), g_E(b)=e \\ [e]  & \text{if }\forall b\in G(E), g_E(b)\neq e \end{cases}
	\]
	Then $h_V$ and $h_E$ describe a morphism of hypergraphs $h\colon H\to Q$ such that $h\circ f=g$. Therefore, $Q$ is injective.
\end{proof}
\begin{corollary}
	Let $Q$ be a hypergraph and $X\defeq Q(V)$. Then $Q$  is injective if and only if $N(Q)$ is injective as a $(X,\Aut(X))$-graph where $N$ is the nerve of the interpretation $I\colon \DD G_{(X,\Aut(X))}\to \C H$ defined above.
\end{corollary}
\begin{proof}
	Let $Q$ be an injective hypergraph. By Proposition \ref{P:Injective}, it is enough to show that for each set map  $j\colon X\to N(Q)(V)$, there is an arc $\alpha\in N(Q)(A)$ such that $\partial(\alpha)=j$. The image of $j$ describes a subset $S$ of vertices in $Q$. Therefore by the result above, there is a hyperedge $e$ with incidence equal to $S$. Let $\alpha\colon I(A)\to Q$ be the arc in $N(Q)$ corresponding to the hypergraph morphism which takes the vertex $x$ to $j(x)$ for each $x\in X$ and the single hyperedge $a\in I(A)$ to $e$. Then $\partial(\alpha)=j$ and thus $N(Q)$ is an injective $(X,\Aut(X))$-graph.
	
	Conversely, suppose $N(Q)$ is injective and let $S\subseteq Q(V)$ a subset of vertices. Let $j\colon X\to N(Q)(V)$ be a set map with image equal to $S$. Then there is an arc $\alpha\in N(Q)(A)$ with incidence $\partial_Q(\alpha)=j$. Since $\alpha$ corresponds to the hypergraph morphism $\alpha\colon I(A)\to Q$, there must be an edge $e\in Q$ such that $a\in I(A)$ is mapped to $e$, i.e., $e$ has incident equal to $S$. Therefore, $Q$ is an injective hypergraph. 
\end{proof}

\subsubsection{The Category of $\underline \Pi_X$-Graphs} \label{S:UnderlinePiGraphs}

Let $X$ and $Y$ be sets. We define the symmetric $X$-power of $Y$, denoted $\underline \Pi_X(Y)$, as the multiple coequalizer of $(\underline \sigma\colon \Pi_X(Y) \to \Pi_X(Y))_{\sigma\in \Aut(X)}$ where $\underline \sigma$ is the $\sigma$-shuffle of coordinates in the product. This definition extends to a functor $\underline{\Pi} _{X}\colon \Set \to \Set$. Note that if $j\colon X'\to X$ is a set map, then there is a natural transformation $\underline{\Pi} _{X}\Rightarrow \underline{\Pi} _{X'}$ induced by the universal mapping property of the product. In particular, when $X\to X'=1$ is the terminal map, we have $\Id_{\Set}=\underline{\Pi} _1\Rightarrow \underline{\Pi} _X$ which we denote by $\eta\colon \Id_\Set\Rightarrow \underline{\Pi} _X$.\footnote{Note that in the case $X=2$, the category of $\underline \Pi_X$-graphs is the category of undirected graphs in the conventional sense in which morphisms are required to map edges to edges.  }

To define an interpretation functor $I\colon \DD G_{(X,\Aut(X))}\to \C G_{\underline \Pi_X}$, we let $q$ be the unordered set $(x)_{x\in X}$ in $\underline \Pi_X(X)$.  Since $\underline \Pi_X(\sigma)(x)_{x\in X}=(x)_{x\in X}$ for each automap $\sigma\colon X\to X$, the interpretation is well-defined.

\begin{proposition}
	The interpretation $I\colon \DD G_{(X,\Aut(X))}\to \C G_{\underline \Pi_X}$ is dense.
\end{proposition}
\begin{proof}
	Let $(E,V,\varphi)$ and $(K,L,\psi)$ be $\C G_{\underline \Pi_X}$-objects and $\lambda\colon D\Rightarrow \Delta(K,L,\psi)$ a cocone on the diagram $D\colon I\ls (E,V,\varphi)\to \C G_{\underline \Pi_X}$. 
	Let $e$ be an edge in $E$ and $f\colon X\to V$ be the set morphism with $\underline{\Pi} _Xf=\varphi(e)$. Then $(\named e, f)\colon I(A)=(1,X,\top)\to (E,V,\varphi)$ is an object in $I\ls (E,V,\varphi)$ and thus there is a morphism $\lambda_{(\named {e},f)}\eqdef(\named{e'},g)\colon D(\named e,f)=(1,X,\top)\to (K,L,\psi)$. By the compatibility of the cocone, this gives us a uniquely defined $h\colon E\to K$, $e\mapsto e'$ on edges. Similarly for each vertex $v\in V$, there is a morphism $(!_E,\named v)\colon I(V)=(\empset, 1,!_1)\to (E,V,\varphi)$ and a cocone inclusion $(!_K,\named w)\colon D(!_E,\named v)=(\empset, 1,!_1)\to (K,L,\psi)$ giving us a factorization on vertices $k\colon V\to L$. Since $\psi\circ h(e)=\underline{\Pi} _X(kf)\circ \top=\underline{\Pi} _X(k)\circ \varphi(e)$ for each edge $E$,  $(h,k)\colon (E,V,\varphi)\to (K,L,\psi)$ is a well-defined $\C G_{\underline \Pi_X}$-morphism which necessarily is the unique factorization of the cocone. Therefore, $I$ is dense. 
	
\end{proof}

\begin{corollary}
	The nerve $N\colon \C G_{\underline \Pi_X}\to \widehat{\DD G}_{(X,\Aut(X))}$ is full and faithful.
\end{corollary}

Note that the realization functor takes a $\widehat{\DD G}_{(X,\Aut(X))}$-object and quotients out the set of arcs by $\Aut(X)$. Hence the unit of the adjunction $\eta_P\colon P\to NR(P)$ is bijective on vertices and surjective on arcs. Hence the adjunction is epi-reflective.

For a $\C G_{\underline \Pi_X}$-object $(B,C,\varphi)$, the embedding given by the nerve functor is given by 
\begin{align*}
N(B,C,\varphi)(V)&=\C G_{\underline \Pi_X}(I(V), (B,C,\varphi))\iso C,\\
N(B,C,\varphi)(A)&=\C G_{\underline \Pi_X}(I(A), (B,C,\varphi))\\ &=\setm{(e,g)\ }{\ e\in B,\ g\colon X\to C\ s.t.\ \underline{\Pi} _Xg=\varphi(e)} 
\end{align*}
The  right-actions are by precomposition, i.e., $(e,g).x=(e,g\circ \named x)$, $(e,g).\sigma=(e,g\circ \sigma)$.

Next, we show that injective and projectives in the category of $\C G_{\underline \Pi_X}$-graphs are precisely those objects which are taken to injective and projective objects in the category of $(X,\Aut(X))$-graphs.

\begin{proposition}\label{P:InjectiveN}
	A $\underline \Pi_X$-graph $Q$ is injective if and only if $N(Q)$ is an injective $(X,\Aut(X))$-graph.
\end{proposition}
\begin{proof}
	If $N(Q)$ is injective, then $Q$ is injective since $N$ is full and faithful and preserves monomorphisms. Conversely, let $Q$ be an injective $\underline \Pi_X$-graph and consider the monomorphism $f\colon G\to H$ and morphism $g\colon G\to N(Q)$ of $(X,\Aut(X))$-graphs. The realization preserves monomorphisms, hence $R(f)\colon R(G)\to R(H)$ is a monomorphism. Since the counit $\varepsilon_Q \colon RN(Q)\to Q$ is an isomorphism, $RN(Q)$ is injective and thus there is a morphism $h\colon R(H)\to RN(Q)$ such that $h\circ R(f)=R(g)$. Therefore, the following diagram commutes
	\[
	\xymatrix@C=.25em{ G \ar@/_5em/[dddr]_{g} \ar@{>->}[rr]^-f \ar[d]^{\eta_G} && H \ar@/^5em/[dddl]^{\overline h} \ar[d]_{\eta_H} \\ NR(G) \ar@{>->}[rr]^-{NR(f)} \ar[dr]_{NR(g)} && NR(H) \ar[dl]^{N(h)} \\ & NRN(Q) \ar[d]_{N(\varepsilon_Q)}^\iso & \\ & N(Q) &}
	\]
	where $\overline h\defeq \varepsilon_Q\circ N(h)\circ \eta_H$. Thus, $\overline h\circ f=g$ and hence $N(Q)$ is injective.
	
\end{proof}

\begin{proposition}\label{P:ProjectiveN}
	A $\underline \Pi_X$-graph $P$ is projective if and only if $N(P)$ is a projective $(X,\Aut(X))$-graph. 
\end{proposition}
\begin{proof}
	If $N(P)$ is projective, then $P$ is projective since $N$ is full and faithful and preserves epimorphisms. Conversely, let $P$ be a projective $\underline \Pi_X$-graph. It is clear that $I(V)$ and $I(A)$ are projective objects in $\C G_{\underline \Pi_X}$, thus $R(\Proj(N(P)))\iso \bigsqcup_{N(P)(V)}I(V)\sqcup \bigsqcup_{N(P)(A)}I(A)$ is projective. Since the projective refinement \\$\Proj(N(P)))\to N(P)$ (see Section \ref{S:InjProj}) is an epimorphism and $\varepsilon_P\colon RN(P)\to P$ is an isomorphism, the composition $R(\Proj(N(P)))\to RN(P)\to P$ is an epimorphism. Thus, $P$ is a split subobject of a coproduct of $I(V)$ and $I(A)$. However, the only split subobjects of such a coproduct is itself a coproduct of $I(V)$ and $I(A)$. Then since $N$ preserves coproducts and $NI(V)=\underline V$ and $NI(A)=\underline A$, $N(P)$ is projective. 
	
\end{proof}

\end{document}